\documentclass[12pt]{amsart}
\usepackage{enumerate}
\usepackage{amsmath}
\usepackage{amssymb}
\usepackage{amsfonts}
\usepackage{amsthm, upref}

    \numberwithin{equation}{section}

\usepackage{bm}
         
         \bmdefine\alphab{\mathbf{\alpha}}
\bmdefine\betab{\mathbf{\beta}}
\bmdefine\sigmab{\mathbf{\sigma}}

\newcommand{\comment}[1]{}
\newcommand{\eq}{\begin{equation}}
\newcommand{\en}{\end{equation}}
\newcommand{\pp}{\mathbb{P}}
\newcommand{\qq}{\mathbb{Q}}
\newcommand{\rr}{\mathbb{R}}
\newcommand{\nn}{\mathbb{N}}

\newcommand{\ev}{\mathbb E}
\newcommand{\dd}{\, \mathrm{d}}
\newcommand{\ddno}{\mathrm{d}}  

\newcommand{\ep}{\hfill $\Box$}

\newtheorem{thm}{Theorem}
\newtheorem{asmp}[thm]{Assumption}

\newtheorem{exm}[thm]{Example}

\newtheorem{rmk}[thm]{Remark}

\newcommand{\field}[1]{\mathbb{#1}}

\newcommand{\setN}{\field{N}}

\newcommand{\setR}{\field{R}}

\DeclareMathAlphabet{\mathpzc}{OT1}{pzc}{m}{it}

\begin{document}

\title[Central limit theorems]{Small time central limit theorems for semimartingales with applications}

\author[S. Gerhold]{Stefan Gerhold}
\address{Vienna University of Technology,
Wiedner Hauptstr. 8/105-1,
A-1040 Austria}
\email{sgerhold@fam.tuwien.ac.at}
\author[M. Kleinert]{Max Kleinert}
\address{Vienna University of Technology and
arithmetica versicherungs- und finanzmathematische Beratungs-GmbH}
\email{kleinert.max@gmail.com}
\author[P. Porkert]{Piet Porkert}
\address{Christian Doppler Laboratory for Portfolio Risk Management, Vienna University of Technology,
Wiedner Hauptstr. 8/105-1,
A-1040 Austria}
\email{piet.porkert@fam.tuwien.ac.at}
\author[M. Shkolnikov]{Mykhaylo Shkolnikov}
\address{Department of Statistics \\ University of California \\ Berkeley, CA 94720-3860}
\email{mshkolni@gmail.com}

\thanks{P.~Porkert gratefully acknowledges financial support from the
Christian Doppler Research Association (CDG) and fruitful collaboration and
   support by Bank Austria
   and COR \& FJA through CDG}

\date{\today}

\begin{abstract}
We give conditions under which the normalized marginal distribution of a semimartingale converges to a Gaussian limit law as time tends to zero. In particular, our result is applicable to solutions of stochastic differential equations with locally bounded and continuous coefficients. The limit theorems are subsequently extended to functional central limit theorems on the process level. We present two applications of the results in the field of mathematical finance: to the pricing of at-the-money digital options with short maturities and short time implied volatility skews.  
\end{abstract}



\maketitle

\section{Introduction}

Limit theorems for finite-dimensional stochastic processes as time goes to infinity have been a classical object of study in probability theory and many results on the existence and uniqueness of invariant distributions, the convergence of the processes to the latter and the limiting behavior of the fluctuations around the limiting distributions have been obtained (see e.g. \cite{Ha80}, \cite{JS03}, \cite{MT92}, \cite{MT93a}, \cite{MT93b} and the references therein). More recently, small time asymptotics of finite-dimensional \textit{continuous time} stochastic processes have attracted much attention. Apart from the theoretical interest, these have become of great importance in various applied fields such as mathematical finance, where the increasingly high frequency of trades in financial markets requires pricing models behaving reasonably both on very short and on long time horizons.  

\bigskip

In the works \cite{ALV07}, \cite{BGM09}, \cite{BC12}, \cite{BBF04}, \cite{FFF10},  \cite{FH09}, \cite{Ja07} and the references therein the authors study the behavior of the random variables $\ev[f(X_{t_0+\delta})|{\mathcal F}^X_{t_0}]$ for small values of $\delta>0$, where $X$ is a finite-dimensional (jump-)diffusion process, a L\'evy process or more generally a semimartingale, $(\mathcal F^X_t)_{t\geq0}$ is the filtration it generates and the function $f$ is taken from a space of suitable real-valued test functions. In \cite{BC12}, this program is carried out for general finite-dimensional semimartingales and under appropriate continuity assumptions on the characteristics of $X$ as well as smoothness assumptions on the function $f$, the almost sure limit 
\eq\label{aslimit}
\lim_{\delta\searrow0} \delta^{-1}\big(\ev[f(X_{t_0+\delta})|{\mathcal F}^X_{t_0}]-f(X_{t_0})\big) 
\en
is determined.   

\bigskip

Here, we are interested in small time Central Limit Theorems for finite dimensional semimartingales; that is, instead of the \textit{almost sure} limit \eqref{aslimit} we are concerned with the limit
\eq\label{distlimit}
\lim_{\delta\searrow0} \delta^{-1/2} \big(f(X_{\delta})-f(X_0)\big)
\en
\textit{in distribution}. More precisely, we give sufficient conditions on the semimartingale $X$ under which, for every suitable test function $f$, the limit \eqref{distlimit} exists and is given by a centered normal random variable (whose variance depends on the particular choice of the function $f$).
The most closely related result in the literature seems to be Theorem~2.5
of Doney and Maller~\cite{DoMa02}, which characterizes the \emph{L\'evy processes}
that satisfy a small time Central Limit Theorem.

\bigskip

In addition to the just described Central Limit Theorems, we prove Functional Central Limit Theorems on the process level and give two applications of our results in the field of mathematical finance: to the pricing of digital options and the asymptotics of implied volatility skews. To outline the first of the two applications, we recall that the price of a digital option with strike $K$ and maturity $\delta$ on an underlying security with price process $X$ in the presence of a constant interest rate $r>0$ is given by the formula
\eq\label{dig_price}
\ev[e^{-r\delta}\,\mathbf{1}_{\{X_\delta>K\}}]=e^{-r\delta}\,\pp(X_\delta>K). 
\en 
In the limit $\delta\searrow0$, that is for short maturities, this price tends to $0$ if $K>X_0$ (\textit{out-of-the-money} options) and to $1$ if $K<X_0$ (\textit{in-the-money} options) as soon as $X$ has right-continuous sample paths. The evaluation of the limit in the case $K=X_0$ (\textit{at-the-money} options) is however much trickier and, in general, the limit can take all values in the interval $[0,1]$ as we show below. However, if a Central Limit Theorem of the type described above holds for the semimartingale $X$, then the limit must be given by $\frac{1}{2}$. Moreover, in a special case we can bound the price in \eqref{dig_price} for any fixed value of $\delta>0$ from above and below by completely explicit functions tending to $\frac{1}{2}$ in the limit $\delta\searrow0$.
By a well known relation between digital prices and implied volatility skews,
we deduce bounds on the latter in certain models with stochastic interest rates.

\bigskip  

For the sake of a cleaner exposition, we first give the assumptions on the semimartingale $X$ and state our main results in the case of continuous trajectories.

\begin{asmp}\label{assumption} Let $T>0$, $x_0\in\setR^m$. Let $X=(X_t^1,\dots ,X_t^m)^\top_{t\in [0,T]}$ be an $\setR^m$-valued continuous semimartingale with canonical decomposition (see e.g.\ page 337 in~\cite{Ka02}) $X-x_0=M+A$, where $M$ is a continuous local martingale, and $A$ has locally finite variation. Assume that
\begin{enumerate}
\item $X_0=x_0$ a.s.;
\item  \label{compensator} there exists an a.s. positive stopping time $\tau_A$ such that a.s.
\begin{equation*}
A_t^j=\int_0^t b^j_s \dd s, \quad t\in[0,\tau_A],\quad j\in\{1,\dots,m\},
\end{equation*}
for an adapted process $b$;
\item \label{boundA} there exists a random variable $C_b$, such that $|b^j_t|\leq C_b<\infty$ for a.e.\ $t\in [0,\tau_A]$ a.s., $j\in\{1,\dots,m\}$;
\item \label{covariation} there exists an a.s.\ positive stopping time $\tau_M$ such that the covariation is a.s.
\begin{equation*}\langle M^j,M^k\rangle_t=\int_0^t\sum_{l=1}^m\sigma^{jl}_s\sigma^{kl}_s
\dd s,\quad t\in[0,\tau_M],\quad j,k\in\{1,\dots ,m\},
\end{equation*} for a progressive process $\sigma$;
\item \label{boundS} there exists a deterministic constant $C_\sigma<\infty$, such that $|\sigma^{jk}_t|\leq C_\sigma$ for a.e.\ $t\in[0,\tau_M]$ a.s., $j,k\in\{1,\dots ,m\}$;
\item as $t\searrow 0$, $\sigma_t\rightarrow L$ a.s., where $L$ is a deterministic $m\times m$-matrix;
\end{enumerate}
\end{asmp}

\medskip

With this notation the Central Limit Theorem and the Functional Central Limit Theorem for continuous semimartingales read as follows. 

\begin{thm}[Central Limit Theorem]\label{stCLT}
Let $X$ satisfy Assumption \ref{assumption}. Then for every $f:\setR^m\rightarrow\setR^n$ such that there exists an open neighborhood $U$ of $x_0$ with $f\in C^2(U,\setR^n)$, we have
\begin{equation*}
\frac{1}{\sqrt{t}}(f(X_t)-f(x_0))\xrightarrow{d}N_f\text{ as }t\searrow 0,
\end{equation*}
where $N_f$ is a normal random vector with mean 0 and covariance matrix
\begin{equation*}
V=(Df)(x_0)L(Df(x_0)L)^\top.
\end{equation*}
Here, $(Df)(x_0)$ stands for the Jacobian of $f$ at $x_0$.
\end{thm}

\begin{thm}[Functional Central Limit Theorem]\label{fCLT}
Let $X$ satisfy Assumption \ref{assumption}. Then for every $f:\setR^m\rightarrow\setR^n$ such that there exists an open neighborhood $U$ of $x_0$ with $f\in C^2(U,\setR^n)$, the processes
\begin{equation*}
Y^{f,u}:=\biggl(\frac{f(X_{ut})-f(x_0)}{\sqrt{u}}\biggr)_{t\in[0,T]},\quad u\in(0,1),
\end{equation*}
converge in law to a Brownian motion with variance-covariance matrix
\begin{equation*}
V=(Df)(x_0)L(Df(x_0)L)^\top.
\end{equation*}
as $u\searrow 0$.
\end{thm}

\medskip

We remark at this point that Assumption \ref{assumption} is satisfied for weak solutions of stochastic differential equations (SDEs) under minimal regularity assumptions on the coefficients.

\begin{rmk}\label{rmk:SDE}
Let $X$ be a weak solution of the $m$-dimensional SDE
\begin{equation*}
\begin{split}
\dd X_t^j &= b_j(t,X_t)\dd t + \sum_{k=1}^{d}\sigma_{jk}(t,X_t)\dd B^k_t,\quad t\geq 0,\quad j\in\{0,\dotsc,m\},\\
X_0 &= x_0\text{ a.s.},
\end{split}
\end{equation*}
where $B$ is a standard $d$-dimensional Brownian motion, $x_0\in\setR^m$, $b:[0,T]\times\setR^m\rightarrow\setR^m$ is uniformly bounded in a neighborhood of $(0,x_0)$ and $\sigma :[0,T]\times \setR^m\rightarrow\setR^{m\times d}$ is continuous in $(0,x_0)$. Then, $X$ satisfies Assumption \ref{assumption} and, hence, Theorems \ref{stCLT} and \ref{fCLT} apply. 
\end{rmk}

\medskip

We also note that, if $X$ satisfies Assumption \ref{assumption} and the matrix $L$ is non-singular, then the price of an at-the-money digital option in \eqref{dig_price} (that is, when $K=x_0$) converges to $\frac{1}{2}$ in the limit $\delta\searrow0$. This result can be significantly sharpened, when $X$ is given by a weak solution of an SDE of the following type. 

\begin{thm}\label{hoe_cont_thm}
Suppose that the process $X$ solves the stochastic differential equation
\begin{align}
\dd X_t &= b(t,\cdot)\dd t+\sigma(t)\dd B_t, \label{sde2}\\
X_0&=x_0, \notag
\end{align}
where $b:\,[0,\infty)\times\Omega\rightarrow\rr^m$ is a bounded predictable process, $\sigma:\,[0,\infty)\rightarrow\rr^{m\times m}$ is a locally square integrable function taking values in the set of invertible matrices such that the smallest eigenvalue of $\sigma(\cdot)^\top\sigma(\cdot)$ is uniformly bounded away from $0$ and $B$ is a standard $m$-dimensional Brownian motion. Then, the bounds
\begin{equation}\label{eq:P bounds}
e^{f_1(t)}\leq \pp\big(X^1_t>X^1_0\big)\leq e^{f_2(t)}, \quad t>0
\end{equation}
apply. Here, the functions $f_1$, $f_2$ are given by
\begin{align*}
  f_1(t)&= -\left(1+\sqrt{\frac{\|\sigma^{-1}b\|_{2,\infty}^2\,t}{2\log 2}}\right)
  \cdot\left(\log 2+\sqrt{\frac{\|\sigma^{-1}b\|_{2,\infty}^2\,t\,\log 2}{2}}\right), \\
  f_2(t)&= -\left(\sqrt{\frac{2\log 2}{\|\sigma^{-1}b\|_{2,\infty}^2\,t}}-1\right)
  \cdot\left(\sqrt{\frac{\|\sigma^{-1}b\|_{2,\infty}^2\,t\,\log 2}{2}}-\frac{1}{2}\,\|\sigma^{-1}b\|_{2,\infty}^2\,t\right),
\end{align*}
where $\|\sigma^{-1}b\|_{2,\infty}=\sup_{t,\omega}|\sigma^{-1}(t)b(t,\omega)|_2$. Moreover, in the limit $t\searrow0$, the functions $e^{f_1}$, $e^{f_2}$ admit the series expansions
\begin{eqnarray}
e^{f_1(t)}&=&\frac{1}{2}- \sqrt{\frac{\log 2}{2}} \|\sigma^{-1}b\|_{2,\infty}\, t^{1/2} +O(t), \\
e^{f_2(t)}&=&\frac{1}{2}+ \sqrt{\frac{\log 2}{2}} \|\sigma^{-1}b\|_{2,\infty}\, t^{1/2} +O(t).
\end{eqnarray}
\end{thm}

\medskip

The rest of the paper is structured as follows. In section \ref{sec_cont}, we give the proofs of Theorems \ref{stCLT}, \ref{hoe_cont_thm} and \ref{fCLT} in this order. In addition, we provide examples of continuous semimartingales, for which the limit in \eqref{dig_price} is not $\frac{1}{2}$ and, therefore, the Central Limit Theorem (Theorem \ref{stCLT}) cannot hold with a non-degenerate Gaussian law in the limit. In section \ref{sec_disc}, we state and prove extensions of Theorems \ref{stCLT}, \ref{hoe_cont_thm} and \ref{fCLT} to semimartingales with jumps. Finally, in section \ref{sec_appl}, we explain the consequences of these results for the prices of at-the-money digital options with short maturities and the small time asymptotics of implied volatility skews.   

\section{Continuous Semimartingales} \label{sec_cont}

We start with the proof of Theorem \ref{stCLT}.

\bigskip

\noindent\textit{Proof of Theorem \ref{stCLT}.} Let~$f$ be as in the statement of the theorem and let~$N_f$ be an ${\mathcal N}(0,V)$ random vector on some probability space $(\widetilde{\Omega},\widetilde{{\mathcal A}}, \widetilde{\pp})$. We need to show
\begin{equation}\label{eq:p1}
\lim_{t\searrow 0}\ev\Bigl[g\Bigl( \frac{f(X_t)-f(x_0)}{\sqrt{t}}\Bigr)\Bigr] = \ev_{\tilde \pp}\bigl[g\bigl(N_f\bigr)\bigr],\quad g\in C_\text{b}(\setR^n,\setR).
\end{equation}
To this end, we fix a function $g\in C_\text{b}(\setR^n,\setR)$, choose an open ball $\boldsymbol{B}$ such that $\overline{\boldsymbol{B}}\subset U$,
and define the hitting time $\overline\tau:=\tau_{\overline{\boldsymbol{B}}^\mathrm{c}}$. Then with
\begin{equation}\label{tau}
\tau:=\overline\tau\wedge\tau_A\wedge\tau_M,
\end{equation}
we have
\begin{eqnarray*}
&&\Big|\ev\Big[g\Big(\frac{f(X_t)-f(x_0)}{\sqrt{t}}\Big)\Big]-\ev_{\tilde\pp}[g(N_f)]\Big|\\
&&\quad\;\;\;\leq\Big|\ev\Big[g\Big(\frac{f(X_t)-f(x_0)}{\sqrt{t}}\Big)-g\Big(\frac{f(X_{t\wedge\tau})-f(x_0)}{\sqrt{t}}\Big)\Big]\Big|\\
&&\quad\quad\;\;\;\;\;\;+\Big|\ev\Big[g\Big(\frac{f(X_{t\wedge\tau})-f(x_0)}{\sqrt{t}}\Big)\Big] -\ev_{\tilde \pp}[g(N_f)]\Big|.
\end{eqnarray*}
Hence in order to show \eqref{eq:p1}, it is sufficent to prove that the two summands in the latter upper bound tend to zero as $t\searrow 0$. Since the event $\{\tau=0\}$ has probability zero, the first summand converges to zero by the Dominated Convergence Theorem. Moreover, the convergence of the second summand to zero will follow, if we can show 
\begin{equation}\label{eq:p2}
\frac{f(X_{t\wedge\tau})-f(x_0)}{\sqrt{t}}\xrightarrow{d} N_f, \qquad t\searrow0.
\end{equation}

\medskip

In order to prove \eqref{eq:p2}, we first note that Doob's Integral Representation Theorem (see e.g.\ Theorem~18.12 on page~358 of~\cite{Ka02}) in combination with part (\emph{\ref{covariation}}) of Assumption \ref{assumption} implies the existence of an $m$-dimensional Brownian motion $B$ (possibly on an extension of the primary probability space) such that a.s.
\begin{equation}\label{DoobIntRep}
M_{t\wedge\tau }^j=\sum_{k=1}^m\int_0^{t\wedge\tau}\sigma^{jk}_s\dd B_s^k,\quad t\in[0,T],\quad j\in\{1,\dots ,m\}.
\end{equation}
By part (\emph{\ref{compensator}}) of Assumption \ref{assumption} and \eqref{DoobIntRep} we therefore have a.s.
\begin{equation}\label{anotherequ}
  X^j_{t\wedge\tau}=x_0+\int_0^{t\wedge\tau}b^j_s
    \dd s+\sum_{k=1}^m\int_0^{t\wedge\tau}\sigma^{jk}_s\dd B_s^k,
    \quad t\in[0,T],\; j\in\{1\dots ,m\}.
\end{equation}
In addition, we recall that, by the Cram\'{e}r--Wold Theorem, \eqref{eq:p2} holds iff for every $s=(s_1,\dotsc,s_n)^{\top}\in\setR^n$
\begin{equation}\label{eq:cw}
\sum_{j=1}^{n}s_j\,\frac{f_j(X_{t\wedge\tau})-f_j(x_0)}{\sqrt{t}} \xrightarrow{d} \sum_{j=1}^{n}s_j\,N_f^j
\end{equation}
as $t\searrow 0$. To show this, we fix $s=(s_1,\dotsc,s_n)^{\top}\in\setR^n$. Applying the local It\^o formula (see e.g.\ Corollary~17.19 on page~341 of~\cite{Ka02}) in combination with \eqref{anotherequ}, we have with $\Psi_t=(\psi^{jk}_t)_{1\leq j,k\leq m}:=\sigma_t\sigma_t^{\top}$ for all $j\in\{1,\dotsc,n\}$:
\begin{equation*}
f_j(X_{t\wedge\tau})-f_j(x_0) = \int_{0}^{t\wedge\tau}(\mathcal L_s f_j)(X_s)\dd s 
+\sum_{k,l=1}^{m}\int_0^{t\wedge\tau}\frac{\partial f_j}{\partial x_l}(X_{s})\,\sigma^{lk}_s\dd B^k_s,\quad  t\geq0,
\end{equation*}
where 
\begin{equation*}
(\mathcal L_s f_j)(u) = \frac{1}{2}\sum_{k,l=1}^{m}\psi_s^{kl}\,\frac{\partial^2 f_j}{\partial x_k\partial x_l}(u)
+\sum_{k=1}^{m}b^k_s\,\frac{\partial f_j}{\partial x_k} (u),\quad u\in U,\quad s\in[0,\tau ].
\end{equation*}
Thus, we have for $t>0$:
\begin{equation}\label{eq:itof}
\begin{split}
\sum_{j=1}^{n} s_j\,\frac{f_j(X_{t\wedge\tau})-f_j(x_0)}{\sqrt{t}} 
&= \frac{1}{\sqrt{t}}\sum_{j=1}^{n}s_j\int_{0}^{t\wedge\tau}({\mathcal L}_s f_j)(X_s)\dd s\\ 
&\mspace{15mu}+\frac{1}{\sqrt{t}}\sum_{j=1}^{n}s_j\sum_{k,l=1}^{m}\int_0^{t\wedge\tau}\frac{\partial f_j}{\partial x_l}(X_{s})\,\sigma_s^{lk}\dd B^k_s.
\end{split}
\end{equation}
By parts (\emph{\ref{boundA}}) and (\emph{\ref{boundS}}) of Assumption \ref{assumption} and the choice of $\boldsymbol{B}$ there exists a random variable $C<\infty$ a.s.\ such that $\sup_{u\in \boldsymbol{B}}|(\mathcal L_s f_j)(u)|\leq C$ a.s.\ for $s\in[0,\tau]$, $j\in\{1,\dotsc,n\}$. Thus, we have 
\begin{equation}\label{anotherequation}
\biggl|\frac{1}{\sqrt{t}}\sum_{j=1}^{m}s_j\int_{0}^{t\wedge\tau}({\mathcal L}_s f_j)(X_s)\dd s\biggr|\leq C\,\sqrt{t}\,\sum_{j=1}^{m}s_j\to 0,\quad t\searrow0\quad\text{a.s.}
\end{equation}
Before examining the second term on the right-hand side of \eqref{eq:itof} we observe that, for every $t\in [0,T]$, the random vector
\begin{equation*}
N_t:=\begin{pmatrix}N_t^1\\\vdots \\N_t^n\end{pmatrix}:
=\begin{pmatrix}
\frac{1}{\sqrt{t}}\sum_{k,l=1}^{m}\frac{\partial f_1}{\partial x_l}(x_0)L_{lk} B_{t}^k\\
\vdots\\
\frac{1}{\sqrt{t}}\sum_{k,l=1}^{m}\frac{\partial f_n}{\partial x_l}(x_0)L_{lk} B_{t}^k
\end{pmatrix}
\end{equation*}
is ${\mathcal N}(0,V)$ distributed. In particular, the distribution of $N_{t}$ is independent of $t$, and $N_t\overset{d}{=}N_f$ for every $t>0$.
With $\xi_{jk}:=\sum_{l=1}^m\frac{\partial f_j}{\partial x_l}(x_0)L_{lk}$ we have for all $h\in C_\text{b}(\setR,\setR)$:
\begin{align*}
&\biggl|\ev\Bigl[h\Bigl(\frac{1}{\sqrt{t}}\sum_{j=1}^{n}s_j\sum_{k=1}^{m}\xi_{jk} B_{t\wedge\tau}^k\Bigr)\Bigr]-\ev_{\tilde \pp}\Bigl[h\Bigl(\sum_{j=1}^{n}s_jN_f^j\Bigr)\Bigr]\biggr|\\
&\mspace{20mu}=\biggl|\ev\Bigl[h\Bigl(\frac{1}{\sqrt{t}}\sum_{j=1}^{n}s_j\sum_{k=1}^{m}\xi_{jk} B_{t\wedge\tau}^k\Bigr)\Bigr] -\ev\Bigl[h\Bigl(\frac{1} {\sqrt{t}}\sum_{j=1}^{n}s_j\sum_{k=1}^{m}\xi_{jk} B_{t}^k\Bigr)\Bigr]\biggr|\\
&\mspace{20mu}\leq \biggl|\ev\Bigl[\Bigl(h\Bigl(\frac{1}{\sqrt{t}}\sum_{j=1}^{n}s_j\sum_{k=1}^{m}\xi_{jk} B_{t\wedge\tau}^k\Bigr)- h\Bigl(\frac{1} {\sqrt{t}}\sum_{j=1}^{n}s_j\sum_{k=1}^{m}\xi_{jk} B_{t}^k\Bigr)\Bigr)\textbf{1}_{\{\tau>t\}}\biggr]\biggr|\\
&\mspace{35mu}+\biggl|\ev\Bigl[\Bigl(h\Bigl(\frac{1}{\sqrt{t}}\sum_{j=1}^{n}s_j\sum_{k=1}^{m}\xi_{jk} B_{t\wedge\tau}^k\Bigr)- h\Bigl(\frac{1} {\sqrt{t}}\sum_{j=1}^{n}s_j\sum_{k=1}^{m}\xi_{jk}B_{t}^k\Bigr)\Bigr)\textbf{1}_{\{\tau\leq t\}}\biggr]\biggr|\\
&\mspace{20mu}\leq 2\|h\|_{\infty}\,\pp(\tau\leq t)\to 0
\end{align*}
as $t\searrow 0$. Therefore, the random variables
\begin{equation}\label{eq:frozen}
\frac{1}{\sqrt{t}}\sum_{j=1}^{n}s_j\sum_{k,l=1}^{m}\int_0^{t\wedge\tau}\frac{\partial f_j}{\partial x_l}(x_0)L_{lk}\dd B_s^k
= \frac{1}{\sqrt{t}}\sum_{j=1}^{n}s_j\sum_{k,l=1}^{m}\frac{\partial f_j}{\partial x_l}(x_0)L_{lk}\, B_{t\wedge\tau}^k
\end{equation}
converge in distribution to $\sum_{j=1}^{n}s_jN_f^j$ as $t\searrow0$. Next, we show that the difference between~\eqref{eq:frozen} and the second term
on the right-hand side of~\eqref{eq:itof} converges to zero in~$L^2$. By the Cauchy--Schwarz inequality and It\^o's isometry we have
\begin{align*}
&\ev\Biggl[\biggl(\sum_{j=1}^{n}s_j\sum_{k,l=1}^{m}\frac{1}{\sqrt{t}}\int_0^{t\wedge\tau}\Bigl(\frac{\partial f_j}{\partial x_l}(X_{s}) \sigma_s^{lk}-\frac{\partial f_j}{\partial x_l}(x_0)L_{lk}\Bigr)\dd B^k_s\biggr)^2\Biggr]\\
&\mspace{10mu}\leq n\,m^2\sum_{j=1}^{n}s_j^2\sum_{k,l=1}^{m}\ev\Biggl[\frac{1}{t}\biggl(\int_0^{t\wedge\tau}\Bigl(\frac{\partial f_j}{\partial x_l}(X_{s}) \sigma_s^{lk} -\frac{\partial f_j}{\partial x_l}(x_0)L_{lk}\Bigr)\dd B^k_s\biggr)^2\Biggr]\\
&\mspace{10mu}=n\,m^2\sum_{j=1}^{n}s_j^2\sum_{k,l=1}^{m}\ev\Biggl[\frac{1}{t}\int_0^{t\wedge\tau}\Bigl(\frac{\partial f_j}{\partial x_l}(X_{s}) \sigma_s^{lk}-\frac{\partial f_j}{\partial x_l}(x_0)L_{lk}\Bigr)^2\dd s\Biggr]\\
&\mspace{10mu}\leq n\,m^2\sum_{j=1}^{n}s_j^2\sum_{k,l=1}^{m}\ev \biggl[ \frac{t\wedge\tau}{t}\max_{s\in[0,t\wedge\tau]}\Bigl(\frac{\partial f_j}{\partial x_l}(X_{s}) \sigma_s^{lk}-\frac{\partial f_j}{\partial x_l}(x_0)L_{lk}\Bigr)^2\biggr],
\end{align*}
which indeed converges to zero as $t\searrow 0$ by the Dominated Convergence Theorem. The just established $L^2$ convergence implies convergence in distribution. Summarizing, we have in the limit $t\searrow 0$:
\begin{align*}
&\biggl|\frac{1}{\sqrt{t}}\sum_{j=1}^{n}s_j\int_{0}^{t\wedge\tau}({\mathcal L}_s f_j)(X_s)\dd s\biggr| \to 0\ \ \text{ a.s.},\\
&\sum_{j=1}^{n}s_j\sum_{k,l=1}^{m}\frac{1}{\sqrt{t}}\int_0^{t\wedge\tau}\frac{\partial}{\partial x_l}f_j(X_{s}) \sigma_s^{lk}\dd B_s^k \xrightarrow{d} \sum_{j=1}^{n}s_jN_f^j,
\end{align*}
so that by Slutsky's theorem \eqref{eq:cw} readily follows. \ep

\begin{rmk}
Recall that, if a family of probability measures satisfies a large deviations principle (LDP) with a rate function $I$, then the validity of a CLT is related to the second derivative of $I$ (see section 1.4 in \cite{deHo00} for a discussion in the case of Cram\'er's theorem). We now (heuristically) outline this connection in a very simple instance of our setup. Suppose that $X$ satisfies a one-dimensional SDE (with zero drift for simplicity)
  \begin{equation*}
    X_t = x_0 + \int_0^t \sigma(X_s)\dd B_s,\quad t\geq 0,
  \end{equation*}
where $\sigma$ is bounded, bounded away from zero and Lipschitz continuous. Then, due to the time-change formalism for one-dimensional diffusions (see e.g.\ Theorem~8.5.1
on page~148 in~\cite{Ok10}), for each $\delta>0$, we can view the random variable $X_\delta$ as the value of the diffusion
  \begin{equation*}
  X^{(\delta)}_t = x_0 + \sqrt{\delta} \int_0^t \sigma(X^{(\delta)}_s)\dd W_s 
  \end{equation*}
at time $1$, where $W$ is the appropriate standard Brownian motion. Now, using the remark following Theorem 5.6.7 in~\cite{DeZe10} on page 214, and the contraction principle (see Theorem 4.2.1 on page 126 of \cite{DeZe10}), we conclude that the random variables $X_t$ satisfy an LDP as $t\searrow0$ with rate function
  \begin{equation*}
  I(x_0+\varepsilon)=\frac{1}{2}\,\inf_{\substack{f\in H^1([0,1]):\\f(0)=x_0,\\f(1)=x_0+\varepsilon}} 
  \int_0^1 \frac{\dot{f}(s)^2}{\sigma(f(s))^2}\dd s. 
 \end{equation*}
Next, let $\Sigma$ be an antiderivative of the function $1/\sigma$. By the assumptions on $\sigma$, the function $\Sigma(f(\cdot))$ belongs to $H^1([0,1])$ if and only if the function $f$ belongs to $H^1([0,1])$. Hence, the latter infimum can be rewritten as
  \begin{equation*}
  \inf_{\substack{v\in H^1([0,1]):\\v(0)=\Sigma(x_0),\\v(1)=\Sigma(x_0+\varepsilon)}} \int_0^1 \dot{v}(s)^2\dd s.
  \end{equation*}
Due to Jensen's inequality, the infimum is reached when $v$ is the affine function connecting $\Sigma(x_0)$ and $\Sigma(x_0+\varepsilon)$. Plugging it in, we end up with  
  \begin{equation}\label{Ibnd}
  I(x_0+\varepsilon)=\frac{1}{2}[\Sigma(x_0+\varepsilon)-\Sigma(x_0)]^2=\frac{1}{2}\left(\int_{x_0}^{x_0+\varepsilon} 
  \frac{\ddno u}{\sigma(u)}\right)^2.
  \end{equation}
That is, for $\varepsilon>0$ small and fixed, we have the asymptotics
  \begin{equation}\label{eq:ldp}
  \pp(X_t \geq x_0 + \varepsilon) \simeq  \exp(-I(x_0+\varepsilon)/t),
  \end{equation}
where $\simeq$ stands for exponential equivalence. Now, \emph{pretend} that we can apply the LDP \eqref{eq:ldp} with a time-dependent $\varepsilon$ defined by $\varepsilon=z\sqrt{t}$, where $z>0$. Since $I(x_0)=I'(x_0)=0$, we have
  \[
    I(x_0 + z\sqrt{t}) = \tfrac12 I''(x_0) z^2 t + o(t),
  \]
and so
  \[
  \pp\left(\frac{X_t-x_0}{\sqrt{t}} \geq z \right) \simeq \exp\left(-\frac{z^2}{2 \sigma(x_0)^2} + o(1)\right).
  \]
The Gaussian limit law is thus correctly identified by this heuristic argument (the case $z<0$ is similar).
\end{rmk}

\medskip

If a semimartingale $X$ satisfies Assumption \ref{assumption}, and the limit law in Theorem \ref{stCLT} is non-degenerate, we clearly have
\begin{equation}\label{anylimit}
\lim_{t\searrow 0} \, \pp(X_t>x_0)=\frac{1}{2}.
\end{equation}
We now give some examples where the value of this limit is not $1/2$.
\begin{exm}
Let us consider the squared Brownian motion $B^2$ in one dimension (no confusion with our superindex convention should arise). Then clearly $\lim_{t\searrow 0}\pp(B_t^2>0)=1$, which does not contradict Theorem \ref{stCLT}. Indeed, the martingale part in the canonical decomposition of $B^2$ is $B_t^2-t=2\int_0^t B_s\dd B_s$, which leads to 
\eq
\langle B_t^2-t\rangle=4\int_0^t B_s^2\dd s \rightarrow 0,\quad t\searrow0 \quad \text{a.s.}
\en
Since all items of Assumption \ref{assumption} are satisfied, Theorem \ref{stCLT} tells us that $\frac{1}{\sqrt{t}}B_t^2$ converges in distribution to a \emph{degenerate} normal random variable.
\end{exm}
\begin{exm}
Denoting by $\Phi$ the standard normal cumulative distribution function, we see that for any $p\in(0,1)$ and a standard Brownian motion $B$, the continuous process $X_t = B_t + \Phi^{-1}(p) \sqrt{t}$ satisfies $\pp(X_t > 0)=p$ for all $t\geq0$.
(Although not related to the present topic, we recall that the process $B_t=W_t+\sqrt{t}$ occurs in Example 3.4 of Delbaen and Schachermayer \cite{DeSc95}. They show that, when used as the price process of a financial security, $X_t$ (and also $\exp(X_t)$) allows for immediate arbitrage; the arbitrage disappears if proportional transaction costs are introduced \cite[Example~4.1]{Gu06}.)
\end{exm}
The following example shows that each probability $p\in[0,1)$ can even be realized by a continuous \emph{martingale}. (Note also that that the \emph{non-continuous} martingale $t-P_t$, where $P_t$ is a Poisson process with parameter $1$, satisfies $\lim_{t\searrow 0} \pp(t-P_t>0)=1$.)

\begin{exm}\label{exm:Bessel}
Consider the squared Bessel process of dimension $\delta\geq0$, that is, the strong solution of the SDE
\begin{equation*}
\dd R^\delta_t=2\sqrt{R^\delta_t}\dd B_t+\delta \dd t
\end{equation*}
with initial value $R^\delta_0=0$. Then, the process $R^\delta_t-\delta t$, $t\geq0$ is a martingale. We claim that
\[
  \text{for all}\ p\in[0,1),\ \text{there is a}\ \delta \in [0,\infty)\
  \text{such that}\ \lim_{t\searrow 0}\,\pp(R^\delta_t-\delta t>0)=p.
\]
The scaling property of squared Bessel processes (see section 1 in chapter XI of \cite{ReYo99}) shows 
\begin{equation*}
\lim_{t\searrow 0} \, \pp(R^\delta_t-\delta t>0)=\pp(R^\delta_1>\delta).
\end{equation*}
We show now that when one varies $\delta$ in $[0,\infty)$, the right-hand side achieves all values $p\in[0,1)$. For $\delta>0$, the random variable $R^\delta_1$ has the gamma distribution with shape parameter $\delta/2$ and scale parameter $2$ (see Corollary 1.4 in section 1 of chapter XI in \cite{ReYo99}). In particular, it has mean $\delta$ and variance $2\delta$. We claim that
\begin{equation}\label{delta_claim}
\lim_{\delta\searrow 0} \, \pp(R^\delta_1>\delta)=0.
\end{equation}
Let $\varepsilon>0$ be arbitrary. By Chebyshev's inequality, we have
\[
  \pp(R^\delta_1 > \delta + \varepsilon) \leq \frac{2\delta}{\varepsilon^2},
\]
and so $\lim_{\delta\searrow 0} \pp(R^\delta_1 > \delta + \varepsilon)=0$. Therefore, recalling that, for any fixed $0<\delta\leq 2$, the density function of $R^\delta_1$ is strictly decreasing, we have the estimates
\begin{align*}
\lim_{\delta\searrow0} \, \pp(R^\delta_1>\delta) & =\lim_{\delta\searrow 0} \, \pp(\delta+\varepsilon>R^\delta_1>\delta) \\
&\leq \varepsilon\lim_{\delta\searrow 0} \, 2^{-\delta/2}\Gamma(\delta/2)^{-1}\delta^{\delta/2-1}e^{-\delta/2}\\
&= \varepsilon\lim_{\delta\searrow0} \, \frac{\delta^{\delta/2-1}}{\Gamma(1+\delta/2)/(\delta/2)}=\frac{\varepsilon}{2}.
\end{align*}
Thus, taking the limit $\varepsilon\searrow 0$, we end up with \eqref{delta_claim}. For $\delta\to\infty$, the random variables $(R_1^\delta-\delta)/(2\delta)^{1/2}$ converge in distribution to a standard normal random variable \cite[p.~340]{JoKoBa94}. This implies
\[
\lim_{\delta\to\infty} \pp(R^\delta_1>\delta)=\frac12.
\]
It now follows from the Intermediate Value Theorem that, for every $p\in[0,\tfrac12)$, we can find a $\delta\geq0$ such that
\begin{equation*}
\lim_{t\searrow 0} \, \pp(R^\delta_t-\delta t>0)=p.
\end{equation*}
(Note that for $\delta=0$, we have $R^\delta_t \equiv 0$, and so $\lim_{t\searrow 0} \, \pp(R^\delta_t-\delta t>0)=0$.) Finally, by considering the martingales $\delta t - R_t^\delta$, we see that all values $p\in[0,1)$ can be achieved.  
\end{exm}
We now take a look at higher order terms beyond the limit in \eqref{anylimit}. If $X_t=B_t+bt$ is a one-dimensional Brownian motion with drift $b\in\rr$, we have $\pp(X_t>x_0)=\tfrac{1}{2}+O(t^{1/2})$. Theorem \ref{hoe_cont_thm}, which we prove now, shows that this
estimate persists for a larger class of It{\^o} processes.
\bigskip

\noindent\textit{Proof of Theorem \ref{hoe_cont_thm}.} Fix a $t>0$ and make a change of probability measure according to the Girsanov Theorem, with the corresponding density being given by
\begin{align*}
\frac{\ddno\qq}{\ddno\pp}:=Z_t^{-1}&:=e^{-N_t-\frac{1}{2}\langle N\rangle_t}\\
&:=\exp\Big(-\int_0^t \sigma(s)^{-1}b(s,\cdot)\dd B_s-\frac{1}{2}\int_0^t |\sigma(s)^{-1}b(s,\cdot)|_2^2\dd s\Big).
\end{align*}
Under $\qq$, the process $X$ solves the equation
\eq
\mathrm{d}X_s=\sigma(s)\dd B^\qq_s
\en 
on $[0,t]$ with initial condition $X_0=x_0$ and where $B^\qq$ is a standard Brownian motion under $\qq$. Thus, $\qq(X^1_s>X^1_0)=\frac{1}{2}$ for all $s\in(0,t]$. Moreover, 
\eq
\pp(X^1_t>X^1_0)=\ev^{\qq}\big[Z_t\,\mathbf{1}_{\{X^1_t>X^1_0\}}\big]. 
\en
To obtain upper and lower bounds on the latter expression, we fix numbers $p,q>1$ such that $p^{-1}+q^{-1}=1$ and apply H\"older's inequality to deduce
\begin{align*}
  \qq\big(X^1_t>X^1_0\big) &= \ev^{\qq}\big[ \mathbf{1}_{\{X^1_t>X^1_0\}}
    Z_t^{1/p} Z_t^{-1/p} \big] \\
  &\leq \ev^{\qq}\big[\mathbf{1}_{\{X^1_t>X^1_0\}} Z_t \big]^{1/p} \, \ev^{\qq}\big[ Z_t^{-q/p} \big]^{1/q}.
\end{align*}
Taking the $p$-th power and rearranging, we get
\begin{align*}
\qq\big(X^1_t>X^1_0\big)^p\,\ev^\qq\big[Z_t^{-q/p}\big]^{-p/q}
&\leq \ev^{\qq}\big[\mathbf{1}_{\{X^1_t>X^1_0\}} Z_t\big]\\
&\leq \qq\big(X^1_t>X^1_0\big)^{1/q}\,\ev^\qq\big[Z_t^p\big]^{1/p},
\end{align*}
where the last upper bound follows again by H\"older's inequality. This can be simplified to
\[
\Big(\frac{1}{2}\Big)^p\,\ev^\qq\big[Z_t^{-q/p}\big]^{-p/q}\leq\pp(X^1_t>X^1_0)\leq\Big(\frac{1}{2}\Big)^{1/q}\,\ev^\qq\big[Z_t^p\big]^{1/p},
\]
or
\eq\label{exp_bnds}
\Big(\frac{1}{2}\Big)^p\,\ev^\pp\big[Z_t^{-q/p-1}\big]^{-p/q} \leq \pp(X^1_t>X^1_0)
\leq\Big(\frac{1}{2}\Big)^{1/q}\,\ev^\pp\big[Z_t^{p-1}\big]^{1/p}.
\en
To estimate the bounds further, we note that
\begin{align}
Z_t^{-q/p-1}
&=e^{-(q/p+1)N_t-\tfrac12(q/p+1)^2\langle N\rangle_t}\cdot
  \exp\Big(\frac{1}{2}\Big(\frac{q}{p}+1\Big)\frac{q}{p}\,\langle N\rangle_t\Big) \notag \\
&\leq e^{-(q/p+1)N_t-\tfrac12(q/p+1)^2\langle N\rangle_t}\cdot
  \exp\Big(\frac{1}{2}\Big(\frac{q}{p}+1\Big)\frac{q}{p}\,t\,\|\sigma^{-1}b\|^2_{2,\infty}\Big)
  \label{eq:low mart bd}
\end{align}
and
\begin{align}
Z_t^{p-1}
&=e^{(p-1)N_t-\tfrac12(p-1)^2\langle N\rangle_t}\cdot
  \exp\big(\tfrac12(p-1)p\,\langle N\rangle_t\big) \notag \\
&\leq e^{(p-1)N_t-\tfrac12(p-1)^2\langle N\rangle_t}\cdot
  \exp\big(\tfrac12(p-1)p\,t\,\|\sigma^{-1}b\|^2_{2,\infty}\big).
   \label{eq:up mart bd}
\end{align}
(Recall that we write $\|\sigma^{-1}b\|_{2,\infty}$ for $\sup_{t,x}|\sigma^{-1}b(t,x)|_2$.) The first factors in \eqref{eq:low mart bd} resp. \eqref{eq:up mart bd} are $\pp$-martingales, since Novikov's condition is satisfied by our assumptions on $b$ and $\sigma$. Therefore, inserting these estimates into \eqref{exp_bnds}, we obtain
\begin{align*}
\sup_{\substack{p^{-1}+q^{-1}=1 \\ p>1}} 
\Big(\frac{1}{2}\Big)^p\,&\exp\Big(-\frac{1}{2}\Big(\frac{q}{p}+1\Big)\,t\,\|\sigma^{-1}b\|^2_{2,\infty}\Big) \\
&\leq \pp(X^1_t>X^1_0) \\
&\leq \inf_{\substack{p^{-1}+q^{-1}=1 \\ p>1}}   
\Big(\frac{1}{2}\Big)^{1/q}\,\exp\big(\tfrac12(p-1)\,t\,\|\sigma^{-1}b\|^2_{2,\infty}\big).
\end{align*}
It is easy to see that the lower bound is maximized by
\[
  p = 1 + \sqrt{\frac{\|\sigma^{-1}b\|^2_{2,\infty}\,t}{2\log 2}}\,,
\]
whereas the upper bound is minimized by
\[
  p= \sqrt{\frac{2\log 2}{\|\sigma^{-1}b\|^2_{2,\infty}\,t}},
\]
which together give \eqref{eq:P bounds}. Finally, the expansions given in the statement of the theorem can be computed by Taylor expansions of the explicit functions in the lower and upper bounds. \ep

\bigskip

We conclude this section with the proof of Theorem \ref{fCLT}. 

\bigskip

\noindent\textit{Proof of Theorem \ref{fCLT}.}
Let $\widetilde{B}$ be a Brownian motion with variance-covariance matrix $V$ and let $(u_l)_{l\in\setN}$ be a sequence with elements in $(0,1)$ such that $u_l\searrow 0$ as $l\rightarrow\infty$. It is sufficient to verify the
convergence of the finite-dimensional distributions
\begin{equation} \label{confindist}
(Y_{t_1}^{f,u_l},\dots ,Y_{t_w}^{f,u_l})\xrightarrow{d} (\widetilde{B}_{t_1},\dots ,\widetilde{B}_{t_w}),\quad t_1,\dots ,t_w\in [0,T],\quad  w\in\setN,
\end{equation}
and the tightness condition
\begin{equation}\label{othercrit}
\lim_{\delta\searrow 0} \, \varlimsup_{l\rightarrow\infty}\pp\Big(\sup_{|s-t|\leq \delta}|Y_s^{f,u_l}-Y_t^{f,u_l}|>\varepsilon\Big)=0,\quad \varepsilon >0.
\end{equation}
Indeed, by Theorem 1.3.2 in \cite{SV06}, condition \eqref{othercrit} implies the tightness of the laws of $Y^{f,u_l}$, $l\in\setN$. Moreover, the convergence \eqref{confindist} allows to the identify the limit points with the law of $\widetilde{B}$. 

\bigskip

First, we focus on \eqref{confindist}. Fix $t_1,\dots,t_w\in [0,T]$ for some $w\in\nn$; then by the Cram\'er--Wold theorem it suffices to show
\begin{equation*}
\sum_{d=1}^w\sum_{j=1}^ns_{dj}(Y^{f,u_l}_{t_d})^j\xrightarrow{d}\sum_{d=1}^w\sum_{j=1}^ns_{dj}\widetilde{B}^j_{t_d}
\end{equation*}
for all $s\in\setR^{w\times n}$ as $l\rightarrow\infty$. Let $\tau$ be defined as in \eqref{tau}. Arguing as in the proof of Theorem~\ref{stCLT}, we see that it is enough to show
\begin{equation*}
\sum_{d=1}^w\sum_{j=1}^ns_{dj}(Y^{f,u_l}_{t_d\wedge\tau})^j\xrightarrow{d}\sum_{d=1}^w\sum_{j=1}^ns_{dj}\widetilde{B}^j_{t_d}
\end{equation*}
as $l\rightarrow\infty$. However, this can be proven analogously to \eqref{eq:cw}.

\bigskip

To show \eqref{othercrit}, note that we may work with the stopped processes $Y_{t\wedge\tau}^{f,u_l}$, $l\in\nn$. Indeed, since $\tau$, as defined in \eqref{tau}, is a.s. positive, we have
\[
  \varlimsup_{l\rightarrow\infty}\pp\Big(\sup_{t\in[0,T]} |Y_t^{f,u_l}-Y_{t\wedge\tau}^{f,u_l}|>\varepsilon)=0,\quad \varepsilon>0. 
\]
The triangle inequality thus shows that \eqref{othercrit} is implied by
\begin{equation}\label{othercritnew}
\lim_{\delta\searrow 0}\,\varlimsup_{l\rightarrow\infty}\pp\biggl(\sup_{|s-t|\leq \delta}|(Y_{s\wedge\tau}^{f,u_l})^j-(Y_{t\wedge\tau}^{f,u_l})^j|>\varepsilon\biggr)=0,\quad \varepsilon>0,\, j\in\{1,\dots ,n\}.
\end{equation}
By \eqref{eq:itof} we get
\begin{multline} \label{glg}
\pp\biggl(\sup_{|s-t|\leq \delta}|(Y_{s\wedge\tau}^{f,u_l})^j-(Y_{t\wedge\tau}^{f,u_l})^j|>\varepsilon\biggr) \leq  \pp\biggl(\sup_{|s-t|\leq\delta}\frac{1}{\sqrt{u_l}}\int_{u_l(s\wedge\tau)}^{u_l(t\wedge\tau)}|({\mathcal L}_r f_j)(X_r)|\dd r >\frac{\varepsilon}{2}\biggr)\\
+  \pp\biggl(\sup_{|s-t|\leq\delta}\bigg|\sum_{k,v=1}^{m} \int_{u_l(s\wedge\tau)}^{u_l(t\wedge\tau)}\frac{\partial f_j}{\partial x_v}(X_r)\sigma_r^{vk}\dd B_r^k\bigg| >\frac{\varepsilon\sqrt{u_l}}{2}\biggr).
\end{multline}
According to \eqref{anotherequation}, we have
\begin{equation*}
\pp\biggl(\sup_{|s-t|<\delta}\frac{1}{\sqrt{u_l}}\int_{u_l(s\wedge\tau)}^{u_l(t\wedge\tau)}|({\mathcal L}_r f_j)(X_r)|\dd r >\frac{\varepsilon}{2}\biggr)\leq \pp\biggl(C\sqrt{u_l}\,\delta >\frac{\varepsilon}{2}\biggr)\xrightarrow{l\rightarrow\infty} 0,\quad \varepsilon >0.
\end{equation*}
We now investigate the second term on the right-hand side of \eqref{glg}. After fixing $\delta$, $j$ and $l$, we define the process 
\[
  F_t:=\sum_{k,v=1}^{m}\,\int_0^{u_l(t\wedge\tau)}\frac{\partial f_j}{\partial x_v}(X_r)\sigma_r^{vk}\dd B_r^k,\quad t\in[0,T]. 
\]
In addition, we introduce the processes
\begin{equation*}
G^{i}_t:= F_{i\delta+t} - F_{i\delta}, \quad t\in I_i:=[0,\delta],\quad i\in\{0,\dots ,\lfloor T/\delta\rfloor-1\},
\end{equation*}
and for $i=\lfloor T/\delta\rfloor$,
\begin{equation*}
G^{\lfloor T/\delta\rfloor}_t:= F_{\lfloor T/\delta\rfloor\delta+t} - F_{\lfloor T/\delta\rfloor\delta},
\quad t\in I_{\lfloor T/\delta\rfloor}:=[0,T-\lfloor T/\delta\rfloor].
\end{equation*}
These are continuous local martingales and, thus, each of them can be represented as a time changed Brownian motion (see e.g.\ Theorem~18.4 on page~352 of~\cite{Ka02}): $G^{i}_t=W_{\langle G^i \rangle_t}^i$. Moreover, the quadratic variation of $G^{i}$ can be bounded according to
\[
  \langle G^{i}\rangle_t \leq \gamma\,C_\sigma^2\,u_l\,\delta ,\quad t\in I_i,\quad i\in \{0,\dots,\lfloor T/\delta\rfloor\},
\]
where $0<\gamma<\infty$ only depends on $m$ and the Jacobian of $f$ on the ball $\overline{\boldsymbol{B}}$ (see the paragraph preceeding \eqref{tau} for the definition of the latter). Now, consider the event $\{\sup_{|t-s|<\delta}|F_t-F_s|>\frac{\varepsilon\sqrt{u_l}}{2}\}$. Clearly, on this event there exist $s_0,\,t_0\in[0,T]$ such that $|s_0-t_0|\leq\delta$ and $|F_{t_0}-F_{s_0}|>\frac{\varepsilon\sqrt{u_l}}{2}$. Without loss of generality we may assume that $0\leq s_0<\delta\leq t_0<2\delta$ (the other cases can be dealt with in the same manner). Then, either $|F_\delta-F_{s_0}|>\frac{\varepsilon\sqrt{u_l}}{4}$, or $|F_{t_0}-F_\delta|>\frac{\varepsilon\sqrt{u_l}}{4}$. In the first case we get
\eq
\frac{\varepsilon\sqrt{u_l}}{4}<|F_\delta-F_{s_0}| \leq |F_{s_0}-F_0|+|F_\delta-F_0|\leq 2\,\sup_{r\in[0,\delta]} |F_r-F_0|.
\en
In the second case we have
\eq
\frac{\varepsilon\sqrt{u_l}}{4}<|F_{t_0}-F_\delta|\leq\sup_{r\in[0,\delta]}|F_{\delta+r}-F_\delta|.
\en
These considerations show that on the event $\{\sup_{|t-s|<\delta}|F_t-F_s|>\frac{\varepsilon\sqrt{u_l}}{2}\}$ there exists an index $i\in\{0,\dots \lfloor T/\delta\rfloor\}$ such that $\sup_{t\in I_i}|G_t^i|>\frac{\varepsilon\sqrt{u_l}}{8}$.
Putting everything together we obtain
\begin{eqnarray*}
&&\pp\biggl(\sup_{|s-t|<\delta}\bigg|\sum_{k,v=1}^{m}\, \int_{u_l(s\wedge\tau)}^{u_l(t\wedge\tau)}\frac{\partial f_j}{\partial x_v}(X_r)\sigma_r^{vk}\dd B_r^k\bigg| >\frac{\varepsilon\sqrt{u_l}}{2}\biggr)\\
&&\quad\leq \pp\biggl(\sup_{t\in I_i}|G^i_t|>\frac{\varepsilon\sqrt{u_l}}{8}\text{ for at least one }i\biggr) \\
&&\quad\leq \sum_{i=0}^{\lfloor T/\delta\rfloor}\pp\biggl(\sup_{t\in I_i}|G^i_t|>\frac{\varepsilon\sqrt{u_l}}{8}\biggr) \\
&&\quad\leq\Bigl(\frac{T}{\delta}+1\Bigr)
\pp\biggl(\sup_{0\leq r\leq \gamma\,C_\sigma^2\,u_l\,\delta}|W_r^i|>\frac{\varepsilon\sqrt{u_l}}{8}\biggr) \\
&&\quad\leq \Bigl(\frac{T}{\delta}+1\Bigr) \exp\biggl(-\frac{\varepsilon^2}{128\,\gamma\,C_\sigma^2\,\delta}\biggr)\xrightarrow{\delta\rightarrow 0}0.
\end{eqnarray*} 
Here, the last estimate follows from Bernstein's inequality
(see e.g.\ Exercise~3.16 on page~153 of~\cite{ReYo99}). We have
established~\eqref{othercritnew} and, thus, the proof is finished.

\ep


\section{Semimartingales with Jumps} \label{sec_disc}

This section is devoted to the extensions of Theorems \ref{stCLT}, \ref{fCLT} and \ref{hoe_cont_thm} to semimartingales with jumps. We start by stating the assumptions on the semimartingale $X$, which will replace Assumption \ref{assumption} when jumps are present. 

\begin{asmp}\label{assjumps}
For a $T>0$, let $X=(X_t^1,\dots ,X_t^m)^\top_{t\in [0,T]}$ be an $\setR^m$-valued semimartingale with decomposition $X=X^c+J$, such that
\begin{enumerate}
\item $X^c$ is a continuous semimartingale satisfying Assumption \ref{assumption};
\item the process $J$ is given by
\[
\quad\quad\quad J_t=\int_0^t\int_{\textbf{B}_1}\psi(s,z)\, \bigl(\Pi (\ddno s,\ddno z)
    -\mu (\ddno s,\ddno z)\bigr)
    +\int_0^t\int_{\setR^m\setminus \textbf{B}_1}\varphi (s,z)\, \Pi (\ddno s,\ddno z),
\]
where $\textbf{B}_1$ denotes the unit ball in $\setR^m$, $\Pi$ is a Poisson random measure on $[0,T]\times\setR^m$ with compensator $\mu$; the $\setR^m$-valued processes $\psi$, $\varphi$ are predictable with respect to the filtration generated by $\Pi$ and
\begin{equation*}
\ev\biggl[\int_0^T\int_{\textbf{B}_1}|\psi (s,z)|^2\, \mu (\ddno s,\ddno z)\biggr]<\infty;
\end{equation*}
\item There exists an a.s. positive stopping time $\tau_J$ such that
\[
\ev\Big[\big|\Pi-\mu\big|\big([0,t\wedge\tau_J]\times\textbf{B}_1\big)\Big]=o(t^{1/2})\quad\text{as}\quad t\searrow0.
\]
\end{enumerate} 
\end{asmp}
We can now formulate the analogue of Theorem \ref{stCLT} in the case of semimartingales with jumps. 

\begin{thm}[Central Limit Theorem with jumps]\label{CLTj}
Let $X$ satisfy Assumption \ref{assjumps}. Then for every $f:\setR^m\rightarrow\setR^n$ such that there exists an open neighborhood $U$ of $x_0$ with $f\in C^2(U,\setR^n)$, we have
\begin{equation*}
\frac{1}{\sqrt{t}}(f(X_t)-f(x_0))\xrightarrow{d}N_f\quad \text{ as }t\searrow 0,
\end{equation*}
where $N_f$ is a normal random vector with mean 0 and covariance matrix
\begin{equation*}
V=(Df)(x_0)L(Df(x_0)L)^\top.
\end{equation*}
\end{thm}
\medskip

\noindent\textit{Proof.} Let $r>0$ be such that the closed ball $\overline{\boldsymbol{B}}_r(x_0)$ with radius $r$ around $x_0$ is contained in $U$. Further, we denote by $\overline{\boldsymbol{B}}_{r/2}(x_0)$ the closed ball with radius $r/2$ around $x_0$ and define the hitting time $\overline{\tau}:=\tau_{\overline{\boldsymbol{B}}_{r/2}(x_0)^c}$. Finally, we introduce the stopping time
\begin{equation} \label{justoti}
\tau:=\overline{\tau}\wedge\tau_A\wedge\tau_M\wedge\tau_J
\end{equation}
and notice that $\tau$ a.s. positive. Therefore, by the same argument as in the proof of Theorem \ref{stCLT}, it suffices to show
\begin{equation*}
\frac{1}{\sqrt{t}}(f(X_{t\wedge\tau})-f(x_0))\xrightarrow{d}N_f \quad \text{as} \quad t\searrow 0.
\end{equation*} 

By It\^{o}'s formula in the form of Proposition 8.19 in \cite{CT04}, we have for all $j\in\{1,\dots,n\}$ and $t\in[0,T]$:
\begin{align}\label{rhs}
f_j(X_{t\wedge\tau})&-f_j(x_0) = \int_{0}^{t\wedge\tau}(\mathcal L_s f_j)(X_s)\dd s 
+\sum_{k,l=1}^{m}\int_0^{t\wedge\tau}\frac{\partial f_j}{\partial x_l}(X_s) \sigma^{lk}_s\dd B^k_s\\
\label{thirdterm}
&+\int_0^{t\wedge\tau}\int_{\boldsymbol{B}_1}\bigl(f_j(X_{s-}+\psi (s,z))-f_j(X_{s-})\bigr)\,\bigl (\Pi(\ddno s,\ddno z)-\mu(\ddno s,\ddno z)\bigr)\\
\label{fourthterm}
&+\int_0^{t\wedge\tau}\int_{\setR^m\setminus\boldsymbol{B}_1}\bigl(f_j(X_{s-}+\varphi (s,z))-f_j(X_{s-})\bigr)\,\Pi (\ddno s,\ddno z).
\end{align}
Arguing as in the proof of Theorem \ref{stCLT}, we see that the vector of terms on the right-hand side of \eqref{rhs}, rescaled by $\frac{1}{\sqrt{t}}$, converges in distribution to $N_f$ as $t\searrow 0$. Thus, the theorem will follow if we can show that the terms \eqref{thirdterm} and \eqref{fourthterm}, rescaled by $\frac{1}{\sqrt{t}}$, converge to zero in probability as $t\searrow 0$.

\bigskip

The term \eqref{thirdterm}, rescaled by $\frac{1}{\sqrt{t}}$, can be decomposed into a sum $T_t^1+T_t^2$ of the following two terms:
\begin{align*}
\frac{1}{\sqrt{t}}\int_0^{t\wedge\tau}\int_{\boldsymbol{B}_1}
\bigl(f_j(X_{s-}+\psi (s,z))-f_j(X_{s-})\bigr)\,\mathbf{1}_{\{|\psi(s,z)|<r/2\}}\,\bigl (\Pi(\ddno s,\ddno z)-\mu(\ddno s,\ddno z)\bigr),\\
\frac{1}{\sqrt{t}}\int_0^{t\wedge\tau}\int_{\boldsymbol{B}_1}
\bigl(f_j(X_{s-}+\psi (s,z))-f_j(X_{s-})\bigr)\,\mathbf{1}_{\{|\psi(s,z)|\geq r/2\}}\,\bigl (\Pi(\ddno s,\ddno z)-\mu(\ddno s,\ddno z)\bigr).
\end{align*}
Then:
\[
\ev\bigl[|T_t^1|\bigr]
\leq\frac{2\|f|_{\overline{\boldsymbol{B}}_r(x_0)}\|_\infty}{\sqrt{t}}\,\ev\Big[\big|\Pi-\mu\big|\big([0,t\wedge\tau_J]\times\textbf{B}_1\big)\Big],
\]
which converges to zero as $t\searrow 0$ by part (3) of Assumption \ref{assjumps}. Moreover, since $J$ a.s. has only finitely many jumps of absolute size greater than $r/2$ on every finite time interval, $T_t^2$ converges to $0$ a.s. as $t\searrow 0$.

\bigskip

Lastly, the term \eqref{fourthterm}, rescaled by $\frac{1}{\sqrt{t}}$, converges to zero a.s. as $t\searrow 0$, since $J$ a.s. has only finitely many jumps of absolute size greater than $1$ on every finite time interval. \ep

\bigskip

As in the case of continuous semimartingales, the Central Limit Theorem can be strengthened to a Functional Central Limit Theorem, which in the presence of jumps reads as follows.

\begin{thm}[Functional Central Limit Theorem with jumps]
Let $X$ satisfy Assumption \ref{assjumps}. Then for every $f:\setR^m\rightarrow\setR^n$ such that there exists an open neighborhood $U$ of $x_0$ with $f\in C^2(U,\setR^n)$, the processes
\begin{equation*}
Y^{f,u}:=\biggl(\frac{f(X_{ut})-f(x_0)}{\sqrt{u}}\biggr)_{t\in[0,T]},\quad u\in(0,1),
\end{equation*}
converge in law to a Brownian motion with variance-covariance matrix given by
\begin{equation*}
V=(Df)(x_0)L(Df(x_0)L)^\top
\end{equation*}
as $u\searrow 0$.
\end{thm}

\medskip

\noindent\textit{Proof.} For each $f$ and $u$ as in the statement of the theorem, we write $Q^{f,u}$ for the law of the process $Y^{f,u}$ on $D([0,T],\setR^n)$, the space of right-continuous functions on $[0,T]$ having left limits; moreover, we denote by $Q^{f,u}_c$ the law of the continuous part of $Y^{f,u}$ on $C([0,T],\setR^n)$. We claim first that the family $(Q^{f,u})_{u\in(0,1)}$ is tight on $D([0,T],\setR^n)$ if and only if the family $(Q^{f,u}_c)_{u\in(0,1)}$ is tight on $C([0,T],\setR^n)$ and, moreover, that the limit points of the two families are the same. 

\medskip

To prove the claim, it suffices to show that for every $\varepsilon>0$ and $j\in\{1,\dots,n\}$:
\begin{equation}\label{eqj}
\pp\Bigl(\sup_{t\in[0,T]}|(J^{f,u}_t)^j|>\varepsilon\Bigr)\rightarrow 0\quad\text{as}\quad u\searrow 0,
\end{equation}
where $J^{f,u}$ denotes the jump part of $Y^{f,u}$. Indeed, if this is the case, then every converging subsequence of $(Q^{f,u})_{u\in(0,1)}$ in $D([0,T],\setR^n)$ corresponds to a converging subsequence of $(Q^{f,u}_c)_{u\in(0,1)}$ in $C([0,T],\setR^n)$ and the limits of the two subsequences have to coincide. Now, since the stopping time defined in \eqref{justoti} is a.s. positive, \eqref{eqj} is implied by
\begin{equation}\label{jumpeqtau}
\pp\Bigl(\sup_{t\in[0,T]}|(J^{f,u}_t)^j|>\varepsilon,\;\tau>uT\Bigr)\rightarrow 0 \quad\text{as}\quad u\searrow 0.
\end{equation}
Furthermore, by It\^{o}'s formula in the form of Proposition 8.19 in \cite{CT04}, we have on the event $\{\tau>uT\}$:
\begin{align}\label{jeqa}
(J^{f,u}_t)^j=& \frac{1}{\sqrt{u}}\int_0^{ut} \int_{\boldsymbol{B}_1}\bigl(f_j(X_{s-}+\psi (s,z))-f_j(X_{s-})\bigr)
\, \bigl(\Pi(\ddno s,\ddno z)-\mu(\ddno s,\ddno z)\bigr)\\
\label{jeqb}&+\frac{1}{\sqrt{u}}\int_0^{ut} \int_{\setR^m\setminus\boldsymbol{B}_1}\bigl(f_j(X_{s-}+\varphi (s,z))-f_j(X_{s-}\bigr)\bigr)
\, \Pi (\ddno s,\ddno z).
\end{align}
As in the proof of Theorem \ref{CLTj}, we decompose the integral on the right-hand side of \eqref{jeqa} according to whether $|\psi(s,z)|<r/2$, or $|\psi(s,z)|\geq r/2$, and call the two resulting processes $(J^{f,u,1})^j$ and $(J^{f,u,2})^j$, respectively. Since the process $(J^{f,u,1})^j$ is obtained by integrating a predictable process with respect to a compensated Poisson random measure, it is a square-integrable martingale. Thus, by Doob's maximal inequality, we have
\begin{multline*}
\pp\Bigl(\sup_{t\in [0,T]}|(J^{f,u,1}_t)^j|>\varepsilon/2,\;\tau>uT\Bigr) \\
\leq \frac{2}{\varepsilon\sqrt{u}}\ev\Big[\Big|\int_0^{(uT)\wedge\tau} \int_{\boldsymbol{B}_1}
\bigl(f_j(X_{s-}+\psi(s,z))-f_j(X_{s-})\bigr)\mathbf{1}_{\{|\psi(s,z)|<r/2\}}\overline{\Pi}(\ddno s,\ddno z)\Big|\Big],
\end{multline*}
where we wrote $\overline{\Pi}$ for $\Pi-\mu$. Moreover, the same argument as in the proof of Theorem \ref{CLTj} shows that the latter upper bound tends to zero as $u\searrow 0$ (by virtue of part (3) of Assumption \ref{assjumps}). Finally, since a.s. the process $J^{f,u}$ has finitely many jumps of size greater than $r/2$ on every finite time interval, the random variables $\sup_{t\in [0,T]}|(J^{f,u,2}_t)^j|$ converge to zero a.s. as $u\searrow 0$. In addition, by the same reasoning, the supremum over $t\in[0,T]$ of \eqref{jeqb} tends to zero a.s. as $u\searrow 0$ as well. Putting everything together, we end up with \eqref{jumpeqtau}, finishing the proof of the claim.

\bigskip

Lastly, one can proceed as in the proof of Theorem \ref{fCLT} to first show the tightness of the family $(Q^{f,u}_c)_{u\in(0,1)}$ on $C([0,T],\setR^n)$ and to subsequently identify each of its limit points with the law of a Brownian motion with variance-covariance matrix $V$. In view of the claim above, this finishes the proof. \ep

\bigskip

We conclude this section by stating and proving the analogue of Theorem \ref{hoe_cont_thm} in the presence of jumps. 

\begin{thm}\label{hoe_jumps_thm}
Suppose that the process $X$ solves the stochastic differential equation
\begin{align}
\mathrm{d}X_t&=b(t,\cdot)\,\mathrm{d}t+\sigma(t)\,\mathrm{d}B_t+\int_{\rr^m} \psi(t,y)\,\Pi(\mathrm{d}t,\mathrm{d}y), \label{sde_drift}\\
X_0&=x_0, \notag
\end{align}
where $b:\,[0,\infty)\times\Omega\rightarrow\rr^m$ is a bounded predictable process with respect to the filtration of the standard $m$-dimensional Brownian motion $B$, $\sigma:\,[0,\infty)\rightarrow\rr^{m\times m}$ is a locally square integrable function taking values in the set of invertible matrices such that the smallest eigenvalue of $\sigma(\cdot)^\top\sigma(\cdot)$ is uniformly bounded away from $0$ and $\psi$ is a predictable process with respect to the filtration of the Poisson random measure $\Pi$. 

Suppose further that $\Pi$ is symmetric with respect to $y$ (so that, in particular, its compensator vanishes) and that $\psi_1(t,y)=-\psi_1(t,y)$ for all $t\geq0$ and $y\in\rr^m$ with probability $1$. Then, the bounds
\eq
e^{f_1(t)}\leq \pp\big(X^1_t>X^1_0\big)\leq e^{f_2(t)}, \quad t>0
\en 
of Theorem \ref{hoe_cont_thm} apply with the same functions $f_1$, $f_2$ as there.
\end{thm}
\medskip

\noindent\textit{Proof.} We start by fixing a $t>0$ and changing the underlying probability measure $\pp$ to an equivalent probability measure $\qq$ according to
\eq
\frac{\mathrm{d}\qq}{\mathrm{d}\pp}
=\exp\Big(-\int_0^t \sigma(s)^{-1}b(s,\cdot)\,\mathrm{d}B_s-\frac{1}{2}\int_0^t |\sigma(s)^{-1}b(s,\cdot)|_2^2\,\mathrm{d}s\Big).
\en
Then, in view of the independence of the continuous and the jump parts of $X$ under $\pp$ and the Girsanov Theorem, the process $X$ solves the stochastic differential equation
\eq
\mathrm{d}X_s=\sigma(s)\,\mathrm{d}B^\qq_s+\int_{\rr^m} \psi(s,y)\,\Pi(\mathrm{d}s,\mathrm{d}y), \quad s\in[0,t] \label{sde_no_drift}
\en
with a standard Brownian motion $B^\qq$ under $\qq$ and initial condition $X_0=x_0$. Moreover, the random variables 
\[
U^{(1)}_t:=\Big(\int_0^t \sigma(s)\,\mathrm{d} B^\qq_s\Big)_1\;\;\;\mathrm{and}\;\;\;
U^{(2)}_t:=\int_0^t\int_{\rr^m} \psi_1(s,y)\,\Pi(\mathrm{d}s,\mathrm{d}y) 
\]
are independent under $\qq$ and their distributions $\eta^{(1)}$ and $\eta^{(2)}$ are symmetric. Hence,
\begin{align*}
\qq(X^1_t>X^1_0)& =\qq\big(U^{(1)}_t+U^{(2)}_t>0\big)\\
&=\int_0^\infty \qq\big(U^{(1)}_t>-c\big)\,\eta^{(2)}(\mathrm{d}c)
+\int_{-\infty}^0 \qq\big(U^{(1)}_t>-c\big)\,\eta^{(2)}(\mathrm{d}c)\\
&=\int_0^\infty \qq\big(U^{(1)}_t>-c\big)\,\eta^{(2)}(\mathrm{d}c)
+\int_{-\infty}^0 1-\qq\big(U^{(1)}_t>c\big)\,\eta^{(2)}(\mathrm{d}c)\\
&=\frac{1}{2}.
\end{align*}
From now on, one can follow the lines of the proof of Theorem \ref{hoe_cont_thm} to finish the proof. \ep

\section{Digital options and the implied volatility slope} \label{sec_appl}

Suppose that the one-dimensional, positive process $S$ models the price of a financial asset, and that $\pp$ is the pricing measure. The riskless rate is $r>0$. The holder of a digital call option with maturity $T$ and strike $K$ receives the payoff $\mathbf{1}_{\{S_T > K\}}$ at maturity. Digital options are peculiar in that the owner receives the full payoff as soon as they are only slightly in the money, as opposed to call options, say, which kick in gradually. By the risk-neutral pricing formula, the value of the digital call at time zero is
\[
  D(K,T) := e^{-rT}\,\ev[\mathbf{1}_{\{S_T>K\}}] =  e^{-rT}\,\pp(S_T>K).
\]
There is a considerable literature on short-maturity approximations for option prices. For OTM (out-of-the-money; $S_0<K$) or ITM (in-the-money; $S_0>K$) digitals, the first order approximation is clear: As soon as the underlying $S$ is a.s. right-continuous at $t=0$, the Dominated Convergence Theorem yields
\begin{align*}
  \lim_{T\to0} D(K,T) =
  \begin{cases}
    0 & \text{if}\ S_0 < K \qquad \text{(OTM)} \\
    1 & \text{if}\ S_0 > K \qquad \text{(ITM).}
  \end{cases}
\end{align*}
Finer information on the OTM decay (which trivially also covers the ITM behavior) comes from small time large deviations principles for the underlying. E.g., see Forde and
Jacquier~\cite{FoJa09} for the case of the Heston model and references about
other diffusion processes.
Our CLT-type results are useful in the \emph{ATM} case $S_0=K$. As an immediate consequence of our limit theorems, we enunciate:
\begin{thm}\label{thm:1/2}
If the process $S$ satisfies the assumptions of Theorem~\ref{CLTj} (in particular, if it satisfies those of Theorem~\ref{stCLT} or Remark~\ref{rmk:SDE}), and the limit law is non-degenerate, then the limiting price of an at-the-money digital call is $1/2$:
\begin{equation}\label{eq:D 1/2}
\lim_{T\to0} D(S_0,T) = \frac12.
\end{equation}
\end{thm}
This (intuitive) result captures virtually all diffusion-based models that have been considered (Black-Scholes, CEV, Heston, Stein-Stein, etc.). Although it seems to be
new in its generality, in particular for jump processes, some special cases can
be inferred from the literature (see the comment at the end of this section).
\bigskip

The jump processes used in financial modelling are often L\'evy processes. It is clear that a compensated compound Poisson process
will yield an (unrealistic) ATM digital price limit of either zero or one (see the remark before Example \ref{exm:Bessel}).
As for the infinite activity case, limit laws are not the appropriate way to get a result like \eqref{eq:D 1/2}. Doney and Maller \cite{DoMa02} have determined all L\'evy processes that admit a short-time CLT, with a criterion involving the tail of the L\'evy measure. While there do exist infinity activity L\'evy processes that satisfy a CLT \cite[Remark~9]{DoMa02}, the L\'evy processes that have been considered in mathematical finance are typically \emph{not} of this kind. For instance, it is easy to see from the characteristic function that the variance gamma process \cite{MaCaCh98} does not admit \emph{any} non-degenerate limit law for $t\to0$, for any normalization. We will discuss these issues further in the more application-oriented companion paper \cite{GeSh12}.

\bigskip

Finally, we discuss the implied volatility skew. Suppose that the underlying $S$ generates the call price surface $C(K,T)$:
\[
  C(K,T) = e^{-rT}\,\ev[(S_T-K)^+], \quad K>0,\quad T>0.
\]
Then the implied volatility (see e.g.~\cite{Le05}) for strike $K$ and maturity $T$ is the volatility $\sigma_{\mathrm{imp}}(K,T)$ that makes the Black-Scholes call price equal to $C(K,T)$:
\[
  C_{\mathrm{BS}}(K,\sigma_{\mathrm{imp}},T) = C(K,T).
\]
The map $K\to\sigma_{\mathrm{imp}}(K,T)$ is called the volatility smile for maturity $T$. It is also called the volatility skew, because it is often
monotone instead of smile-shaped, but we will reserve the term skew for the \emph{derivative} $\partial_K \sigma_{\mathrm{imp}}(K,T)$.
If $C(K,T)$ is smooth in $K$, it equals (we omit arguments)
\[
  \partial_K \sigma_{\mathrm{imp}} =
  -\frac{\partial_K C_{\mathrm{BS}}-\partial_K C}{\partial_\sigma C_{\mathrm{BS}}}.
\]
Under mild assumptions (e.g., if the law of $S_T$ is absolutely continuous), we have
\begin{equation}\label{eq:diff call}
  \partial_K C = - e^{-rT}\,\pp(S_T \geq K) = -D(K,T),
\end{equation}
from which we deduce the (well-known) connection between the volatility skew and the price of a digital call (see e.g.~\cite{Ga06}):
\[
  \partial_K \sigma_{\mathrm{imp}} =
  -\frac{D(K,T)+\partial_K C_{\mathrm{BS}}}{\partial_\sigma C_{\mathrm{BS}}}.
\]
Inserting the explicit Black-Scholes vega and digital
price (see e.g.~\cite{MuRu05}), we obtain
\[
  \partial_K \sigma_{\mathrm{imp}} =
  \frac{-D(K,T)+\Phi(-\sigma_{\mathrm{imp}}\sqrt{T}/2)}{K \sqrt{T}\,n(\sigma_{\mathrm{imp}}\sqrt{T}/2)},
\]
with $\Phi$ and $n$ denoting the standard normal cdf and density, respectively. For $T\to0$, we have $\sigma_{\mathrm{imp}}\sqrt{T}=o(1)$
under the following mild assumptions \cite[Proposition 4.1]{RoRu09}: 
\begin{align}
  &(S_0-K)^+ \leq C(K,T) \leq S_0 \quad \text{(no arbitrage bounds),} \label{eq:no arb} \\
  & \lim_{T\to0} C(K,T) = (S_0-K)^+, \label{eq:ass lim} \\
  & T \mapsto C(K,T) \quad \text{is non-decreasing.} \label{eq:ass dec}
\end{align}
Therefore,
\begin{equation}\label{eq:slope asympt}
  \partial_K \sigma_{\mathrm{imp}}
  \sim \frac{\sqrt{2\pi}}{K \sqrt{T}} \left(\frac12 -D(K,T)
      - \frac{\sigma_{\mathrm{imp}}\sqrt{T}}{2\sqrt{2\pi}}
      + O((\sigma_{\mathrm{imp}}\sqrt{T})^3) \right), \quad T\to0.
\end{equation}
We see that the small time behavior of the skew is related to that of the digital price. At the money, the latter will typically tend to $1/2$
(see Theorem \ref{thm:1/2}), and so higher order estimates are needed to get the first order asymptotics of the ATM skew $\partial_K \sigma_{\mathrm{imp}}|_{K=S_0}$. To this end, we apply our Theorem \ref{hoe_cont_thm}, and compare our findings with the standard
model free slope bounds \cite[page 36]{FoPaSi00} 
\begin{equation}\label{eq:std bounds}
  -\frac{\sqrt{2\pi}}{S_0 \sqrt{T}}(1-\Phi(d_2)) e^{-rT + d_1^2/2}
  \leq \frac{\partial \sigma_{\mathrm{imp}}}{\partial K}
  \leq \frac{\sqrt{2\pi}}{S_0 \sqrt{T}}\Phi(d_2) e^{-rT + d_1^2/2},
\end{equation}
where
\begin{align*}
  d_1 &= \frac{\log(S_0/K) + (r+
    \tfrac12 \sigma^2_{\mathrm{imp}})T}{\sigma_{\mathrm{imp}}\sqrt{T}}, \\
  d_2 &= d_1 - \sigma_{\mathrm{imp}} \sqrt{T}.
\end{align*}
Such bounds can give guidance on model choice; recall that
the market slope seems to grow like $T^{-1/2}$ for short maturities~\cite{AlLePoVi08}.
Note that the following result accomodates stochastic interest rates,
and recall that we assume in this section that the dimension is $m=1$ . 
Under stochastic interest rates, the digital call price is
\begin{equation}\label{eq:stoch r dig}
  D(K,T) = \ev\big[e^{-\int_0^T r(s) \dd s}
    \mathbf{1}_{\{S_T>K\}} \big].
\end{equation}
To calculate the implied volatility, a \emph{deterministic} rate~$r$ has
to be chosen (e.g., by $e^{-rT} = \ev[\exp(-\int_0^T r(s) \dd s)]$).
This choice is irrelevant for Theorem~\ref{thm:slope bounds}, though.
\begin{thm}\label{thm:slope bounds}
Assume that the price process satisfies the SDE
\[
\dd S_t/S_t = r(t) \dd t + \sigma(t) \dd B_t
\]
with the stochastic short rate process $(r(t))_{t\geq0}$,
and that the log-price $X=\log S$, whose drift is $b(t)=r(t)-\tfrac12\sigma^2(t)$, satisfies the assumptions of Theorem \ref{hoe_cont_thm}.
Assume further that $\partial_K C(K,T)=-D(K,T)$ holds (cf.~\eqref{eq:diff call}) and that~\eqref{eq:ass lim} and~\eqref{eq:ass dec} are satisfied. Then we have the ATM slope bounds
\begin{align*}
  \partial_K \sigma_{\mathrm{imp}}|_{K=S_0}
  &\geq \frac{\sqrt{2\pi}}{K \sqrt{T}} \left( -C\sqrt{T} 
    - \frac{\sigma_{\mathrm{imp}} \sqrt{T}}{2 \sqrt{2\pi}}
    + O(T) + O((\sigma_{\mathrm{imp}} \sqrt{T})^3) \right), \\
  \partial_K \sigma_{\mathrm{imp}}|_{K=S_0}
  &\leq \frac{\sqrt{2\pi}}{K \sqrt{T}} \left( C\sqrt{T} 
    - \frac{\sigma_{\mathrm{imp}} \sqrt{T}}{2 \sqrt{2\pi}}
    + O(T) + O((\sigma_{\mathrm{imp}} \sqrt{T})^3) \right), 
\end{align*}
where
\[
  C = \sqrt{\frac{\log 2}{2}} \|\sigma^{-1}b\|_{2,\infty}.
\]
\end{thm}
\begin{proof}
According to~\eqref{eq:stoch r dig}, the ATM digital price equals
\[
   D(S_0,T) = \ev\big[e^{-\int_0^T r(s) ds}\,\mathbf{1}_{\{X_T>x_0\}} \big].
\]
The discount factor is $1+O(T)$, so we can apply Theorem \ref{hoe_cont_thm} to conclude
\[
  \frac12 - C \sqrt{T} + O(T) \leq D(S_0,T) \leq  \frac12 + C \sqrt{T} + O(T).
\]
Now, the result follows from \eqref{eq:slope asympt}. Note that $\sigma_{\mathrm{imp}}\sqrt{T}=o(1)$ by \cite[Proposition~4.1]{RoRu09}, since we assume \eqref{eq:ass lim} and \eqref{eq:ass dec}, and \eqref{eq:no arb} is satisfied in our setup.
\end{proof}

\medskip

The bounds in Theorem~\ref{thm:slope bounds} are asymptotically stronger than
the general estimate~\eqref{eq:std bounds},
which is of order $O(T^{-1/2})$, since $\sigma_{\mathrm{imp}} \sqrt{T}=o(1)$. If the Berestycki-Busca-Florent formula~\cite{BeBuFl02} holds, then implied volatility tends to a constant. Therefore, our bounds are \emph{considerably} stronger than~\eqref{eq:std bounds} in this case, namely of order $O(1)$. The models covered by Theorem~\ref{thm:slope bounds}
thus do \emph{not} match the empirical slope behavior $T^{-1/2}$, similarly to
stochastic volatility models~\cite{Le00}, whose slope also behaves like $O(1)$.

To conclude our discussion of ATM digitals and the implied volatility skew,
note that, for some diffusion processes, the result in Theorem~\ref{thm:1/2}
is implicitly in the literature. To wit, by~\eqref{eq:slope asympt},
a non-exploding ATM slope requires a limit price of $1/2$ of the digital.
See Durrleman~\cite[page~59]{Du04} for a general expression for the implied
volatility slope that shows that it does not explode, e.g., in the Heston model.

\bibliographystyle{siam}
\bibliography{clt}

\end{document}